\documentclass[a4paper,10pt]{amsart}
\usepackage{txfonts,amsfonts}
\usepackage[arrow,matrix]{xy}
\usepackage{amsmath,amssymb,amscd,bbm,amsthm,mathrsfs,dsfont,bm}
\usepackage{extarrows}
\usepackage{latexsym}
\usepackage{graphicx}
\usepackage{color}
\usepackage{xy}
\usepackage{leftidx}
\input xy
\input xypic
\xyoption{all}

\numberwithin{equation}{section}

\theoremstyle{plain}
\newtheorem{theorem}{Theorem}[section]
\newtheorem{corollary}[theorem]{Corollary}
\newtheorem{lemma}[theorem]{Lemma}

\newtheorem{proposition}[theorem]{Proposition}

\theoremstyle{definition}
\newtheorem{definition}[theorem]{Definition}

\theoremstyle{remark}
\newtheorem{remark}[theorem]{Remark}

\DeclareMathOperator{\lw}{\rm LW}

\DeclareMathOperator{\gkdim}{\rm GKdim }

\DeclareMathOperator{\sh}{\rm Sh }

\DeclareMathOperator{\lex}{\rm lex }
\DeclareMathOperator{\glex}{\rm glex }

\DeclareMathOperator{\gr}{\rm gr }

\newcommand{\B}{\mathcal{B}}

\newcommand{\N}{\mathcal{N}}

\renewcommand{\L}{\mathbb{L}}

\begin{document}

\title{The structure of connected (graded) Hopf algebras revisited}

\author{C.-C. Li,\quad G.-S. Zhou}

\address{\rm Li \newline \indent
School of Mathematics and Statistics, Ningbo University, Ningbo 315211, China
\newline \indent E-mail: thesilencesea@gmail.com
\newline\newline
\indent Zhou \newline \indent
School of Mathematics and Statistics, Ningbo University, Ningbo 315211, China
\newline \indent E-mail: zhouguisong@nbu.edu.cn}

\begin{abstract}
Let $H$ be a connected graded Hopf algebra over a field of characteristic zero and $K$ an arbitrary graded Hopf subalgebra of $H$. We show that there is a family of homogeneous elements of $H$ indexed by a totally order set that
satisfy several desirable conditions, which reveal interesting connections between $H$
and $K$. In particular, the set on non-decreasing products of these elements give a basis
for $H$ as a left and right graded $K$-module. As one of its consequences, we see that $H$ is a graded iterated Hopf Ore extension of $K$ of derivation type provided that $H$ is of finite Gelfand-Kirillov dimension. The main tool of this work is Lyndon words, along the idea developed by Lu, Shen and the second-named author in \cite{ZSL3}.
\end{abstract}




\maketitle





\section*{Introduction}

\newtheorem{maintheorem}{\bf{Theorem}}
\renewcommand{\themaintheorem}{\Alph{maintheorem}}
\newtheorem{maincorollary}[maintheorem]{\bf{Corollary}}
\renewcommand{\themaincorollary}{}
\newtheorem{mainquestion}[maintheorem]{\bf{Question}}
\renewcommand{\themainquestion}{}


Throughout the paper, we work over a fixed field denoted by $k$ which is of characteristic zero. All vector spaces, algebras, coalgebras and unadorned tensors  are over $k$.

A \emph{connected Hopf algebra} is a Hopf algebra with coradical of dimension one. Recall that the coradical of a Hopf algebra is defined to be the sum of all of its simple subcoalgebras.  Though the condition of being connected is quite restrictive, it has rich examples  from diverse fields of mathematics. Two well-known classes of examples are the universal enveloping algebras of Lie algebras and the coordinate rings of unipotent algebraic groups, which are cocommutative and commutative respectively.  Another celebrated class of examples appear in combinatorics, including the  Hopf algebras of symmmectric functions, of quasi-symmetric functions and of permutations \cite{GR, ABS, AS}. These Hopf algebras tend to have bases naturally parametrized by combinatorial objects (such as partitions, compositions, permutations, etc.), and they are always of infinite Gelfand-Kirillov dimension (GK dimension, for short). Recently, connected Hopf algebras of GK dimensions up to  $4$ were completely classified by Wang, Zhang and Zhuang   \cite{Zh,WZZ}, when the base field is algebraically closed. It is a part of the longstanding project of classifying Hopf algebras that satisfy some finiteness conditions.  Some  new interesting connected Hopf algebras occur in the classification results.

The motivation of this paper is to understand the structure of connected Hopf algebras. We do it by considering the subclass of \emph{connected graded Hopf algebras}, which are Hopf algebras $H$ equipped with a grading $H=\bigoplus_{n\geq 0}H_n$  such that  $H$ is simultaneously a graded algebra and a graded coalgebra,  the antipode  preserves the given grading and moreover $H_0$ is one dimensional.  It is shown in \cite[Theorem A]{ZSL3} that every connected graded Hopf algebra has  \emph{PBW generators} which satisfy some desirable conditions. By PBW generators  of an  algebra $A$ we mean a family of elements of $A$ indexed by a totally ordered set such that the nondecreasing product of these indexed elements form a linear basis of $A$. In this paper, we prove a relative version of \cite[Theorem A]{ZSL3} as follows.

\begin{maintheorem}\label{maintheorem-generator}(Theorem \ref{generator-connected-graded-Hopf-algebra})
	Let $H$ be a connected graded Hopf algebra and $K$ a graded Hopf subalgebra. Then there is an indexed family $(z_\gamma)_{\gamma\in \Gamma}$ of homogeneous elements of $H$ and a total order $\leq$ on $\Gamma$ such that
	\begin{enumerate}
		\item $\Delta(z_\gamma)  \in 1\otimes z_\gamma +z_\gamma\otimes 1 + (H_+^{<\gamma} \otimes H_+^{<\gamma})_{\deg(z_\gamma)}$ for every index  $\gamma \in \Gamma$, where $H^{<\gamma}$ is the subalgebra of $H$ generated by $K$ and $\{~ z_\delta~|~ \delta \in \Gamma,~ \delta<\gamma ~\}$;
		\item $[z_\gamma , H^{<\gamma}] \subseteq H^{<\gamma}$ for every index  $\gamma \in \Gamma$, where $[-,-]$ denotes the usual commutator;
		\item $\{~ z_{\gamma_1} \cdots z_{\gamma_n} ~|~ n\geq0,~ \gamma_1,\ldots, \gamma_n \in \Gamma, ~ \gamma_1\leq \cdots \leq \gamma_n ~\}$ is a homogeneous basis of $H$  as a left graded $K$-module as well as a right graded $K$-module;
		\item $\gkdim H = \gkdim K + \#(\Gamma)$, where $\#(\Gamma)$ is the number of elements of $\Gamma$.
	\end{enumerate}
\end{maintheorem}

The above theorem has some interesting consequences. For example, it follows that every commutative connected (graded) Hopf algebra is a (graded) polynomial algebra over any of its (graded) Hopf subalgebras in some family of (graded) variables, see Corollary \ref{polynomial-generator} and Corollary \ref{connected-commutative-Hopf-algebra}. As special cases, some classical results are reconfirmed: (1) every commutative affine connected Hopf algebrna is a polynomial algebra  \cite[Theorem 0.1]{BGZ}; (2) the shuffle algebras are graded polynomial algebras \cite[Theorem 3.1.1]{Rad1}; (3)  the Hopf algebra of quasi-symmetric functions is a graded polynomial algebra over the  Hopf algebra of symmetric functions \cite[Corollary 2.2]{MaRe}. We expect more applications of Theorem \ref{maintheorem-generator} for the study of group schemes and combinatorial Hopf algebras.

Another interesting consequence of Theorem \ref{maintheorem-generator} concerns Hopf Ore extensions. This line of research begun in \cite{Pan} and has attracted a number of studies recently \cite{BOZZ, Hu, BZ2, ZSL3}. Recall that an \emph{$n$-step iterated Hopf Ore extension} (\emph{n-step IHOE}, for short) of a Hopf algebra $R$ is a Hopf algebra $H$ which contains $R$ as a Hopf subalgebra and there is a chain of Hopf subalgebras $R=H^0\subset H^1 \subset \cdots \subset H^n= H$ with each of the extensions $H^i\subset H^{i+1}$ being  an Ore extension as algebras (Definition \ref{definition-Hopf-Ore}). According to \cite[Theorem B]{ZSL3}, a connected graded Hopf algebra of GK dimension $d$ can be obtained as a $d$-step IHOE of the base field. In this paper, we generalize \cite[Theorem B]{ZSL3} as follows.

\begin{maintheorem}\label{maintheorem-Hopf-Ore}(Theorem \ref{IHOE-graded})
Let $H$ be a connected graded Hopf algebra of GK dimension $d$  and $K$ a graded Hopf subalgebra of GK dimension $e$. Then $H$ is a graded $(d-e)$-step  IHOE of $K$ of derivation type (Definition \ref{definition-Hopf-Ore}).
\end{maintheorem}

Note that connected Hopf algebras of finite GK dimension are not necessarily IHOEs' of the base field. However, the known counterexamples are all the universal enveloping algebras of certain Lie algebras \cite[Example 3.1 (iv)]{BOZZ}. By the classification results obtained in \cite{Zh, WZZ}, connected Hopf algebras of GK dimensions up to $4$ are all IHOEs' of the universal enveloping algebra of its primitive space. In view of the above theorem, it is natural to ask the following question.

\begin{mainquestion}
Let $H$ be a connected Hopf algebra of finite GK dimension. Is it necessarily that $H$ is an IHOE of the universal enveloping algebra of its primitive space?
\end{mainquestion}

Now let us sketch the proof of Theorem \ref{maintheorem-generator}. First we choose a set $X$ of homogeneous generators for $H$ such that $X':= X\cap K$ generates $K$, and then fix a well order on $X$ satisfying certain conditions related to $X'$. Denote by $I$ (resp. $J$) the ideal of defining relations  on $X$ (resp. $X'$) for  $H$ (resp. $K$). In addition, let $\N_I$ (resp. $\N_J$) be the set of $I$-irreducible (resp. $J$-irreducible) Lyndon words on $X$ (resp. $X'$). By the combinatorial properties of Lyndon words and the standard bracketing on words, one may construct from $X$ a new family of homogeneous generators $(z_\gamma)_{\gamma\in \N_I}$ for $H$. The coalgebra structure of $H$  assures that this new family of generators satisfy some desirable conditions (Proposition \ref{quasi-Lie-imply-LK}). Note that several key parts of this  observation were already obtained in \cite[Proposition 2.4]{ZSL3}, of which the arguments have benefited from some of the ideas in \cite{Kh, Uf}. The decisive new observation for this work is the equality $\N_J= \N_I\cap \langle X'\rangle$, due to the feature of Lyndon words (Lemma \ref{Lyndon-compare}). In Theorem \ref{generator-connected-graded-algebra} we combine Proposition \ref{quasi-Lie-imply-LK} and Lemma \ref{Lyndon-compare}, as well as some other ideas, to conclude that the family $(z_\gamma)_{\gamma\in \Gamma}$, where $\Gamma:= \N_I\backslash \N_J$, together with the restriction of the lexicographic order to $\Gamma$ satisfy the requirements.

The paper is organized as follows. In section 1, we recall the definition and basic facts of Lyndon words. In section 2, we study the structure of connected graded algebras that admit pseudo-comultiplications (Definition \ref{definition-pseudo-comultiplication}) by taking advantage of the excellent  combinatorial properties of Lyndon words. This section is the technical part of the paper and as a matter of fact, Theorem \ref{maintheorem-generator} is a direct consequence of Theorem \ref{generator-connected-graded-algebra}. The final section studies the structure of connected (graded) Hopf algebras. In particular, Theorems  \ref{maintheorem-generator} and \ref{maintheorem-Hopf-Ore} are proved there.

\section{Preliminaries}

\label{subsection-Lyndon}
In this section, we introduce some basic notations that will be used in the sequel, recall 
the definition of Lyndon words and  present some of their well-known properties for reader's convenience. The  material on Lyndon words is mostly extracted or modified  from \cite{GF, Kh, Lo,ZSL3}.

Throughout this section, $X$ stands for a well-ordered alphabet.  If $X=\{x_1,\ldots, x_\theta\}$ for some positive integer $\theta$, then $X$ is tacitly equipped with the natural order $x_1<\cdots <x_\theta$. 
We denote by $\langle X\rangle $ the set of all words on $X$. The empty word is denoted  by $1$.  The length of a word $u$ is denoted by $|u|$.

Let $u, \, v$ be two words on $X$. We call $v$ a \textit{factor}  of  $u$  if there exist words $w_1,w_2$ on $X$ such that  $w_1vw_2=u$. If $w_1$ (resp. $w_2$) can be chosen to be the empty word then $v$ is called a \emph{prefix} (resp. \emph{suffix}) of $u$.  A  factor (resp. prefix, resp. suffix) $v$ of $u$ is called  \textit{proper} if $v\neq u$. The \emph{lexicographic order} on $\langle X\rangle$, denoted by $\leq_{\rm lex}$, is defined as follows. For words $u,v\in \langle X\rangle$,
\begin{eqnarray*}
	\label{definition-deglex}
	u <_{\rm lex} v\,\, \Longleftrightarrow \,\, \left\{
	\begin{array}{llll}
		v \text{ is a  proper prefix of } u,\quad \text{or}&& \\
		u=rxs,\, v=ryt,\, \text{ with } x,y\in X;\, x<y;  \, r,s,t\in \langle X\rangle.
	\end{array}\right.
\end{eqnarray*}
Clearly, it is a total order compatible with the concatenation of words from left but not from right.  For example if $x,y\in X$ with $x>y$, one has $x>_{\rm lex}x^2$ but $xy<_{\rm lex} x^2y$. 

The next result is well-known, we include a proof for readers' convenience.

\begin{lemma}\label{lex-stablize}
	For any integer $n\geq 1$, the restriction of $\leq_{\lex}$ on the set of all words on $X$ of degree less than or equal to $n$ is a well order.
\end{lemma}

\begin{proof}
	It suffices to see every nonincreasing sequence  $u_1\geq_{\lex} u_2\geq_{\lex} \cdots$ of words of degree $\leq n$ stabilizes eventually. We prove it by induction on $n$. If $u_r\neq 1$ for some $r$ then $u_s\neq 1$ for $s\geq r$, so we may assume first that all $u_r$ are nonempty. Since $\leq_{\lex}$ coincides with $\leq$  on $X$, the result is true for $n=1$. For $n\geq 2$, consider the decompositions $u_r=v_rw_r$ for $r\geq 1$, where $v_r\in X$. Clearly, $$v_1\geq v_2\geq v_3\geq \cdots$$
	by the definition of $\leq_{\lex}$. Since $\leq$ is a well order on $X$, there exists $p\geq 1$ such that $v_p=v_{p+1}=\cdots.$ So we get a nonincreasing sequence $w_p\geq_{\lex} w_{p+1}\geq_{\lex}  \cdots$ of words of degree $\leq n-1$. By induction, there exists $q\geq p$ such that $w_q= w_{q+1}=\cdots$, hence  $u_q= u_{q+1}=\cdots$. This completes the proof.
\end{proof} 

\begin{definition}
	A word $u\in \langle X\rangle$ is called \textit{Lyndon}  if $u$ is nonempty and $u>_{\rm lex} wv$ for every factorization $u=vw$ with $v,w$ nonempty. The set of all Lyndon words on $X$ is denoted by 
	$\L=\mathbb{L}(X)$.
\end{definition}

For a word  $u$ of length $\geq2$, define   $u_R\in \langle X\rangle $ to be  the lexicographically largest proper suffix of $u$ and define  $u_L\in \langle X\rangle $ by the decomposition $u=u_Lu_R$.  The pair of words
$$\sh(u):=(u_L,u_R)$$ is called  the {\em Shirshov factorization} of $u$. As an example,  $\sh(x_2^2x_1x_2x_1)=(x_2^2x_1,x_2x_1)$.

Lyndon words enjoy many excellent combinatorial properties, we collect below some of them for reader's interests. Due to the different flavors of the definition of the lexicographic order in literature, the  statements that we present are adjusted accordingly from the given references.

\begin{proposition}\label{fact-Lyndon} \hfill
	\begin{enumerate}
		\item[(L1)] (\cite[Proposition 5.1.2]{Lo}, \cite[Lemma 2]{Kh}) A nonempty word $u$ is Lyndon if and only if it is lexicographically greater than any one of  its proper nonempty suffixes.
		\item[(L2)]  (\cite[Proposition 5.1.3]{Lo} $\&$ \cite[Lemma 4.3 (3)]{GF}) Let $u=w_1w_2,\, v=w_2w_3$ be Lyndon words. If $u >_{\rm lex} v$ (which holds in priori when $w_2\neq1$), then $w_1w_2w_3$ is a Lyndon word.
		\item[(L3)] (\cite[Proposition 5.1.3]{Lo}) Let $u$ be word of length $\geq2$. Then $u$ is a Lyndon word if and only if $u_L$ and $u_R$ are both Lyndon words and $u_L>_{\rm lex} u_R$.
		\item[(L4)](\cite[Proposition 5.1.4]{Lo}) Let $u,v$ be Lyndon words. Then $\sh(uv)=(u,v)$ if and only if either $u$ is a letter or  $u$ is of length $\geq 2$ with $u_R\leq_{\rm lex} v$.
		\item[(L5)](\cite[Theorem 5.1.5]{Lo}) Every nonempty word $u$ can be written uniquely as a (lexicographically) nondecreasing product $u=u_1u_2\cdots u_r$ of Lyndon words.
		\item[(L6)] (\cite[Lemma 4.5]{GF}) Let $u=u_1u_2\cdots u_r$ be a nondecreasing product of Lyndon wordsx. If $v$ is a Lyndon factor of $u$ then $v$ is a factor of some $u_i$. \hfill $\Box$
	\end{enumerate}
\end{proposition}

Let $k\langle X\rangle$ be the free algebra on $X$, which has all words on $X$ as a linear basis. Elements of $k\langle X\rangle $ are frequently  called (noncommutative) polynomials on $X$ (with coefficients in $k$). In the sequel, let us assume that  the free algebra $k\langle X\rangle$ is connected graded with each letter homogeneous of positive degree. The degree of a homogeneous polynomial $f$ is denoted by $\deg(f)$.  

We need some elementary notions from (noncommutative) Groebner basis theory. The \emph{graded lex order} on $\langle X\rangle$, denoted by $\leq_{\glex}$, is defined as follows. For words $u, \, v$, 
\begin{eqnarray*}
	\label{definition-deglex}
	u <_{\rm \glex} v\,\, \Longleftrightarrow \,\, \left\{
	\begin{array}{llll}
		\deg(u)<\deg(v), \quad \text{or} && \\
		\deg(u) =\deg(v)\, \text{ and } u<_{\lex} v.
	\end{array}\right.
\end{eqnarray*}
Clearly, it is an \emph{admissible order} on $\langle X\rangle$, i.e. a well order on $\langle X\rangle$ that is compatible with  concatenation of words from both sides.  The \emph{leading word} of a nonzero polynomial $f$, denoted by $\lw(f)$, is the largest word that occurs in $f$ (with respect to $\leq_{\glex}$). The coefficient of $\lw(f)$ in $f$ is called the \emph{leading coefficient} of $f$. For any set $F\subseteq k\langle X\rangle$, we write $\lw(F):=\{~\lw(f)~|~f\in F, \; f\neq 0~\}.$ 

Let $I$ be an ideal of $k\langle X\rangle$. A word on $X$ is called \emph{$I$-reducible}  if it is the leading word of some polynomials in $I$. A word that is not $I$-reducible is called \emph{$I$-irreducible}. An \emph{obstruction} of $I$ (with respect to $\leq_{\glex}$) is an $I$-reducible word whose proper factors are all $I$-irreducible. We denote by $\N_I$ the set of all $I$-irreducible Lyndon words, by $\mathcal{O}_I$ the set of all obstructions of $I$, and by $\B_I$ the set of all lexicographically nondecreasing product of Lyndon words in $\N_I$.

\begin{lemma}\label{Lyndon-obstruction}
	Let $I$ be an ideal of $k\langle X\rangle$. The set of all $I$-irreducible words is included in $\B_I$, and these two sets are equal if and only if $\mathcal{O}_I$ consists of Lyndon words.
\end{lemma}

\begin{proof}
	It is an easy consequence of Proposition \ref{fact-Lyndon} (L5, L6) and the fact that every $I$-reducible word has a factor in $\mathcal{O}_I$.
\end{proof}

\section{Pseudo-comultiplications on connected graded  algebras}

The primary interest of this work is the structure of connected (graded) Hopf algebras.  The  techniques employed or  developed in this paper are aimed to deal with it. It turns out that the coassociativity of the comultiplication plays no role in some key observations. In this section, we study the structure of connected graded algebras $A$ which admit algebra homomorphisms $A\to A\otimes A$ that satisfy some features of the comultiplication of connected graded Hopf algebras. One of the main tools is the theory of Lyndon words reviewed in the previous section. 

Recall that a \emph{connected graded algebra} is an algebra $A$ on which a decomposition $A=\bigoplus_{n\geq 0}A_n$ is specified such that $A_0= k\, 1_A$ and $A_i\cdot A_j\subseteq A_{i+j}$ for $i,j\geq 0$. 
We write $A_+:= \bigoplus_{n\geq 1}A_n$ in the sequel.

\begin{definition}\label{definition-pseudo-comultiplication}
	Let $A=\bigoplus_{n\geq0} A_n$ be a connected graded algebra. A \emph{pseudo-comultiplication} on $A$ is a homomorphism of graded algebras $\Delta:A\to A\otimes A$ such that 
	\[
	\Delta(a)\in 1\otimes a+ a\otimes 1+(A_+\otimes A_+)_{\deg(a)} \label{TC} 
	\] 
	for every homogeneous element $a$ of positive degree. Clearly, to check the above condition it suffices to verify it on any family of homogeneous generators of $A$.
\end{definition}

The whole section is devoted to prove the following result. 

\begin{theorem}\label{generator-connected-graded-algebra}
	Let $A$ be a connected graded algebra and $B\subseteq A$ a graded subalgebra. Let  $\Delta:A\to A\otimes A$ be a pseudo-comultiplication on $A$ such that $\Delta(B) \subseteq B\otimes B$. 
	Then there is an indexed family $(z_\gamma)_{\gamma\in \Gamma}$ of homogeneous elements of $A$ of positive degree and a total order $\leq$ on $\Gamma$ such that
	\begin{enumerate}
		\item $\Delta(z_\gamma)  \in 1\otimes z_\gamma +z_\gamma\otimes 1 + (A_+^{<\gamma} \otimes A_+^{<\gamma})_{\deg(z_\gamma)}$ for every index  $\gamma \in \Gamma$,
		where $A^{<\gamma}$ is the subalgebra of $A$ generated by $B$ and $\{~ z_\delta~|~ \delta \in \Gamma,~ \delta<\gamma ~\}$;
		\item $[z_\gamma , A^{<\gamma}] \subseteq A^{<\gamma}$ for every index  $\gamma \in \Gamma$, where $[-,-]$ denotes the usual commutator;
		\item $\{~ z_{\gamma_1} \cdots z_{\gamma_n} ~|~ n\geq0,~ \gamma_1,\ldots, \gamma_n \in \Gamma, ~ \gamma_1\leq \cdots \leq \gamma_n ~\}$ is a homogeneous basis of $A$  as a left  graded $B$-module as well as a right graded $B$-module;
		\item  $\gkdim A = \gkdim B + \#(\Gamma)$, where $\#(\Gamma)$ is the number of elements of $\Gamma$.
	\end{enumerate}
\end{theorem}

The proof of this theorem will be addressed  after some preparatory results. 

\begin{remark}
The totally ordered set $(\Gamma, \leq)$ desired in Theorem \ref{generator-connected-graded-algebra} may have no smallest element. Let $A$ be the universal enveloping algebra of the  Lie algebra with basis $\{e_n\}_{n\geq 1}$ and Lie bracket  $[e_i,e_j] = (j-i)e_{i+j}$. Then $A$ is a connected graded algebra with $e_n$ of degree $n$. Take $B$ to be the trivial subalgebra $k\cdot 1_A$. Then the family $(e_n)_{n\in \Gamma}$ with $\Gamma=\{1,2,3,\ldots\}$ and the total order $\leq$ on $\Gamma$ given by $1 > 2> 3 > \cdots$ satisfy the requirements. Clearly, in this case, $(\Gamma, \leq)$ has no smallest element.
\end{remark}

Let us proceed to show Theorem \ref{generator-connected-graded-algebra}. We continue to use the notations and conventions employed in Section \ref{subsection-Lyndon}. So $X$ stands for a well-ordered alphabet and the free algebra $k\langle X\rangle$ is connected graded with letters homogeneous of positive degrees. In addition to the Lyndon words, another brick to construct the required family of generators in Theorem \ref{generator-connected-graded-algebra} is the so-called \emph{standard bracketing} on $\langle X\rangle$, which is the map $[-]: \langle X\rangle \to k\langle X\rangle$ defined as follows.  First set $[1] =1$ and  $[x]:=x$ for  $x\in X$; and then for  words $u$ of length $\geq 2$, inductively  set 
\begin{eqnarray*}
	[u]=
	\left\{
	\begin{array}{ll}
		~ [[u_L], [u_R]], & u \text{ is Lyndon},  \\
		~ [u_L]  [u_R], &  u \text{ is not Lyndon}.
	\end{array}
	\right.
\end{eqnarray*}
Here, $[-,-]: k\langle X\rangle \times k\langle X\rangle \to k\langle X\rangle$ stands for the commutator operation on $k\langle X\rangle$ defined by $[f,g] = fg-gf.$ Note that $[u_1u_2\cdots u_n] = [u_1][u_2]\cdots [u_n]$ for Lyndon words $u_1\leq_{\rm lex}u_2\leq_{\rm lex} \cdots \leq_{\rm lex} u_n$.  

Recall from \cite[Definition 2.1]{ZSL3} that a (graded) \emph{triangular comultiplication} on $k\langle X\rangle$ is a homomorphism  $\Delta:k\langle X\rangle \to k\langle X\rangle \otimes k\langle X\rangle$ of graded algebras such that  
\[
\Delta(x)  ~ \in ~  1\otimes x + x\otimes 1 +   k\langle X_{<x} \rangle_+ \otimes k\langle X_{<x} \rangle_+ 
\]
 for each letter $x\in X$, where $X_{<x}:=\{~x'\in X~| ~ x'<x~\}$. It is clear that every triangular comultiplication on $k\langle X\rangle$ is a  pseudo-comultiplication on $k\langle X\rangle$. The converse does not hold. However,  if $x<x'$ follows necessarily from $\deg(x)<\deg(x')$ for each pair of letters $x, x'\in X$, then every pseudo-comultiplication on $k\langle X\rangle$ is a triangular comultiplication on $k\langle X\rangle$. 
For a word $w$ and an ideal $I$ of $k\langle X\rangle$, we denote by $k\langle X|I\rangle^{<w}$  (resp.  $k\langle X|I\rangle^{\leq w}$)  the subalgebra of $k\langle X\rangle$ generated by  
$$ \{~ [u]~|~ u\in \N_I, ~ u<_{\lex}  w ~\} \quad \Big (  \text{ resp. } \{~ [u]~|~ u\in \N_I, ~ u \leq_{\lex} w ~\}  ~ \Big).$$
Note that they are all graded subalgebras of $k\langle X\rangle$.


\begin{proposition}\label{quasi-Lie-imply-LK}
	Let $I$ be a homogeneous proper ideal of $k\langle X\rangle$ and $A:=k\langle X\rangle/ I$. Assume that there exists a  triangular comultiplication  $\Delta$ on $k\langle X\rangle$ such that $\Delta(I) \subseteq k\langle X\rangle \otimes I + I \otimes k\langle X\rangle$. 
	\begin{enumerate}
		\item $\B_I$ equals to the set of $I$-irreducible words, and  $\mathcal{O}_I$  consists of Lyndon words.
		\item $	[v] \in  k\langle X|I\rangle^{<v}_{\deg(v)}  +I$ for every $I$-reducible Lyndon words $v$.
		\item $	\Delta([v]) \in 1\otimes [v] + [v]\otimes 1 +\big((k\langle X|I\rangle_+^{<v}+I)\otimes (k\langle X|I\rangle_+^{<v}+I)\big)_{\deg(v)}$ for every $v\in \N_I$.
		\item $[u][v] -[v] [u] \in k\langle X|I\rangle_{\deg(uv)}^{\leq uv}  +I$ for every pair $u,v\in \N_I$ with $u>_{\lex} v$.
		\item $\{~ [w]+I\; |\; w\in \B_I ~ \}$ is a homogeneous basis of $A$.
		\item $\gkdim A=\#(\N_I)$, where $\#(\N_I)$ is the number of elements of $\N_I$.
	\end{enumerate}
\end{proposition}

\begin{proof}
	Parts (2), (4) and (5) are obtained in \cite[Proposition 2.4]{ZSL3}. By Part (5) and the fact that $\lw([u]) =u$ for any word $u$ on $X$, it is not hard to see that $\{~ w+I\; |\; w\in \B_I ~ \}$ is a basis of $A$.  By Lemma \ref{Lyndon-obstruction}, Part (1) follows.
	By \cite[Proposition 2.3]{ZSL3}, one has 
	\[
	\Delta([v]) \in 1\otimes [v] +[v] \otimes 1+ (k\langle X\rangle_+^{<v} \otimes k\langle X\rangle_+^{<v})_{\deg(v)},
	\]
	where $k\langle X\rangle^{<v}$ is the subalgebra of $k\langle X\rangle$ generated by 
	$\{\,[u]\;|\; u<_{\lex}v, \, u\in \L ~\}.$ Note that 
	\[
	k\langle X|I\rangle_+^{<v} +I= k\langle X\rangle_+^{<v}+I
	\]
	by Part (2), one then obtains Part  (3). Finally we show Part (6). For any finite subset $\Xi$ of $\mathcal{N}_I$, let $A_\Xi$ be the subalgebra of $A$  generated by $\{~ [v]+I ~|~ v \in \Xi~\}$. For any integer $n$, let $d_\Xi(n)$ be the number of nondecreasing sequences $(v_1,\ldots, v_p)$ in $\Xi$ such that  $\deg(v_1 \cdots v_p)=n$. Then
	\begin{equation*}
		\sum_{n\geq0} \dim(A_{\Xi})_n ~ t^n \geq \sum_{n\geq 0} d_\Xi(n) ~ t^n= \prod_{v \in \Xi} \big (1-t^{\deg (	v)}  \big )^{-1}
	\end{equation*}
	by Part (5) and a simple combinatorial argument, where  the inequality means that the difference of the series  has no negative coefficients. Consequently,
	\begin{equation}\label{Hilbert-GKdim}
		\gkdim(A_\Xi) = \limsup_{n\to \infty} \log_n \sum_{i\leq n} \dim(A_\Xi)_i \geq \limsup_{n\to \infty} \log_n \sum_{i\leq n} d_\Xi(i)= \#(\Xi). \tag{$*$}
	\end{equation}
	Here, the first equality is by \cite[Lemma 6.1 (b)]{KL} and the second equality is by \cite[Proposition 2.21]{ATV2}.
	It follows that if $\mathcal{N}_I$ is infinite then $\gkdim(A) \geq \gkdim(A_\Xi) \geq \#(\Xi)$ for any finite subset $\Xi$ of $\mathcal{N}_I$ and hence  $\gkdim(A) =\infty$. If $\mathcal{N}_I$ is finite then $A=A_{\mathcal{N}_I}$ and the inequality in ($*$)  becomes an equality for $\Xi=\mathcal{N}_I$ by Part (5), whence $\gkdim(A) = \gkdim(A_{\mathcal{N}_I}) = \#(\mathcal{N}_I)$.
\end{proof}

\begin{remark}
	Connected graded algebras of the form $A=k\langle X\rangle/ I$ with  $\mathcal{O}_I$ consisting of Lyndon words are studied in  \cite{GF2,Zhou,ZL1}, where some important homological and algebraic invariants of $A$ are described in terms of $\N_I$. One interesting source of such algebras are Artin-Schelter regular algebras of global dimension  $\leq5$ that satisfy some extra conditions \cite[Theorem 8.1]{ZL}. Proposition \ref{quasi-Lie-imply-LK} provides another natural source of such algebras, including all connected graded Hopf algebras. 
\end{remark}

\begin{lemma}\label{irreducible-Lyndon-subalgebra}
	Let $I$ be a homogeneous ideal of $k\langle X\rangle$ and $A:=k\langle X\rangle/ I$. Assume that there exists a  triangular comultiplication  $\Delta$ on $k\langle X\rangle$ such that $\Delta(I) \subseteq k\langle X\rangle \otimes I + I \otimes k\langle X\rangle$. 
	Then for every sequence $(w_1,\ldots, w_m)$ of Lyndon words in $\N_I$, the coset  $[w_1]\cdots [w_m] +I$ in $A$  is a linear combination of cosets of the form $[v_1]\cdots [v_p] +I$ with $v_1,\ldots, v_p\in \N_I$ satisfying
	\[
	v_1 \leq_{\lex} \cdots \leq_{\lex} v_p\leq_{\lex} \max\{w_1,\ldots, w_m\} \quad \text{and} \quad\deg(v_1\cdots v_p) = \deg(w_1\cdots w_m), 
	\]
	where $\max\{w_1,\ldots, w_m\}$ is the lexicographically maximal word among $w_1,\ldots, w_m$.
\end{lemma}

\begin{proof}
	Let $N_I^{(\infty)}$ be the set of  finite sequences  in $\N_I$, including the empty sequence. Let  $\leq$ be the partial order on $\N_I^{(\infty)}$ defined as follows.  For $ a=(a_1,\ldots, a_m)$ and $b=(b_1,\ldots, b_n)$ in $\N_I^{(\infty)}$,  $a< b$ if and only if  $\sum_{i=1}^m\deg(a_i) = \sum_{j=1}^n\deg(b_j)$ but  there  exists an integer $p\leq \min\{m,n\}$ such that $a_i=b_i$ for $i<p$ and $a_p<_{\rm lex} b_p$.  Clearly, $\leq$ is  compatible with the left and right concatenation of finite sequences. In addition, $\leq$ satisfies the descending chain condition by Lemma \ref{lex-stablize}. 
	Now we proceed to show the result by induction on $(w_1,\ldots, w_m)$ with respect to $\leq$. If $(w_1,\ldots, w_m)$ is nondecreasing, there is nothing to prove. Otherwise, assume $w_{i}>_{\rm lex} w_{i+1}$ for some $i$. By Proposition \ref{quasi-Lie-imply-LK} (4),  
	\begin{eqnarray*}
		[w_1]\cdots [w_m]
		&\in&  [w_1]\cdots [w_{i-1}] [w_{i+1}][w_i] [w_{i+2}] \cdots [w_m] \\
		&& + \sum_{(c_1,\cdots, c_s)} k \cdot [w_1]\cdots [w_{i-1}] [c_1]\cdots [c_s][w_{i+2}] \cdots [w_m] + I,
	\end{eqnarray*}
	where $(c_1,\ldots, c_s)$ runs over finite sequences in $\N_I$ such that
	\[
	\deg(c_1\cdots c_s) = \deg(w_iw_{i+1}) \quad \text{ and } \quad \max\{c_1,\ldots, c_s\}   \leq_{\lex} w_iw_{i+1} <_{\lex}  \max
	\{w_1,\ldots, w_m\}.
	\] 
	It follows that $\max\{w_1,\ldots, w_{i-1}, c_1, \ldots, c_s, w_{i+2},\ldots, w_m\} \leq_{\lex} \max \{ w_1,\ldots, w_m\}$ and 
	\begin{eqnarray*}
		(w_1,\ldots, w_{i-1}, c_1, \ldots, c_s, w_{i+2},\ldots, w_m)  < (w_1,\ldots, w_m).
	\end{eqnarray*} 
	Also we have $$(w_1,\ldots, w_{i-1}, w_{i+1}, w_i, w_{i+2},\ldots, w_m) < (w_1,\ldots, w_m).$$
	By the induction hypothesis, $[w_1]\cdots [w_m]+I$ has the required decomposition.
\end{proof}

The following  result is a crucial observation for the proof of Theorem \ref{generator-connected-graded-algebra}. First let us recall more terminology from standard Groebner basis theory. A set $F\subseteq k\langle X\rangle$ is called \emph{reduced} (with respect to $\leq_{\glex}$) if $0\not\in F$ and for every $f\in F$, the leading coefficient of $f$ is $1$ and all words occurring in $f$ have no factor in $\lw(F\backslash\{f\})$.
For an ideal $I$  of $k\langle X\rangle$, a \emph{Groebner basis} of $I$ (with respect to $\leq_{\glex}$) is a subset $\mathcal{G} \subseteq I$ such that every $I$-reducible word has a factor in $\lw(\mathcal{G})$. It is well-known that every ideal admits a unique reduced Groebner basis \cite[Proposition 14]{Nord}. Note that the above-mentioned definitions may be formalized in the same way with respect to any admissible order on $\langle X\rangle$. Nevertheless, in this paper we only employ the graded lex order on $\langle X\rangle$ to apply the method of Groebner basis theory.

\begin{lemma}\label{Lyndon-compare}
	Let $I$ be a homogeneous proper ideal of $k\langle X\rangle$, $X'$ a subset of $X$ and $J= I\cap k\langle X'\rangle $. Assume  that $x_1<x_2$ for every $x_1\in X'$ and $x_2\in X\backslash X'$, and assume there is a triangular comultiplication $\Delta$ on $k\langle X\rangle$ such that $\Delta(k\langle X'\rangle) \subseteq k\langle X'\rangle \otimes k\langle X'\rangle$ and  $\Delta(I) \subseteq k\langle X\rangle \otimes I + I \otimes k\langle X\rangle$. Then 
	\begin{eqnarray*}
		\mathcal{N}_J= \mathcal{N}_I\cap \langle X'\rangle, \quad 
		\mathcal{O}_J= \mathcal{O}_I \cap \langle X'\rangle \quad \text{and} \quad 
		\mathcal{G}_J= \mathcal{G}_I\cap k\langle X'\rangle,
	\end{eqnarray*}
	where $\mathcal{G}_I$ (resp. $\mathcal{G}_J$) is the reduced Groebner basis of $I$ (resp. $J$) with respect to $\leq_{\glex}$.
\end{lemma} 

\begin{proof}
	First we show $\mathcal{N}_J= \mathcal{N}_I\cap \langle X'\rangle$. Let $v$ be a Lyndon word on $X'$. Clearly, if $v$ is $J$-reducible then it is $I$-reducible. Conversely, assume $v$ is $I$-reducible. Since the biggest letter in a Lyndon word must occur in the left-most place, Lyndon words on $X$  lexicographically smaller than $v$ are necessarily Lyndon words on $X'$. It is then clear that $v$ is $J$-reducible  by Proposition \ref{quasi-Lie-imply-LK} (2). 
	
	Next we show $\mathcal{O}_J= \mathcal{O}_I \cap \langle X'\rangle$. By Proposition \ref{quasi-Lie-imply-LK} (1),  $\mathcal{O}_I$ and $\mathcal{O}_J$ consist of Lyndon words. Note that a word on $X$ (resp. $X'$) is $I$-irreducible (resp. $J$-irreducible) if and only if  its Lyndon factors are all so.  For any Lyndon word $u$ on $X'$, 
	\begin{eqnarray*}
		u \in \mathcal{O}_J
		&\Longleftrightarrow& u \text{ is $J$-reducible and its proper Lyndon factors are all $J$-irreducible}\\
		&\Longleftrightarrow& u \text{ is $I$-reducible and its proper Lyndon factors are all $I$-irreducible}\\
		&\Longleftrightarrow& u\in \mathcal{O}_I.
	\end{eqnarray*}
	Here, the first and the third equivalences are clear, and the second equivalence is by  $\mathcal{N}_J= \mathcal{N}_I\cap \langle X'\rangle$ shown above. The desired equality follows.
	
	Finally, we show $\mathcal{G}_J= \mathcal{G}_I\cap k\langle X'\rangle$. Since $\lw(\mathcal{G}_I)= \mathcal{O}_I$ and $\lw(\mathcal{G}_J)= \mathcal{O}_J$ by standard Groebner basis theory, they both consist of Lyndon words  and $\lw(\mathcal{G}_J) = \lw(\mathcal{G}_I) \cap \langle X'\rangle$. For any word $u\in \lw(\mathcal{G}_J)$, let $f_u$ (resp. $g_u$) be the unique element of $\mathcal{G}_J$ (resp. $\mathcal{G}_I$) with leading word $u$. By the definition of $\mathcal{G}_J$, all words that occur in $f_u-u$ have no factor in $\lw(\mathcal{G}_J)$, and hence they have no factor in $\lw(\mathcal{G}_I)$. Thus $f_u-u$ is a linear combination of $I$-irreducible words. Consequently,  $f_u=g_u$ for any $u\in \lw(\mathcal{G}_J)$  and the desired equality follows directly. 
\end{proof}

\begin{remark}
	For a proper ideal $I$ of $k\langle X\rangle$ and a subset $X' \subset X$, knowing a generating set $G$ for $I$, it is not necessarily true that $G'=G\cap k\langle X'\rangle$ is a generating set for $J=I\cap k\langle X'\rangle$. To make it happen, a well-known choice of $G$ is an arbitray Groebner basis of $I$ with respect to an arbitrary admissible order $\sigma$ on $\langle X\rangle$ that eliminates $X\backslash X'$ in the sense of \cite[Definition 17]{Nord}. Indeed, $G'$ is  a Groebner basis of $J$ with respect to the restriction of $\sigma$ to $\langle X'\rangle$ (see \cite[Proposition 19]{Nord}), hence $G'$ is a generating set of $J$. Nevertheless, Lemma \ref{Lyndon-compare} provides  more choices of $G$ when  $I$ satisfy specific conditions.
\end{remark}

Now we are ready to prove Theorem \ref{generator-connected-graded-algebra}.

\begin{proof}[Proof of Theorem \ref{generator-connected-graded-algebra}]
	By choosing  a set of homogeneous generators of $A$ of positive degree which contains a generating set of $B$, we may fix a set $X$ of graded variables, a decomposition $X=X'\cup X''$, and a homogeneous ideal $I$ of $k\langle X\rangle$ such that $A= k\langle X\rangle/I$ as graded algebras, the cosets of variables in $X'$ form a generating set of $B$ and the cosets of variables in $X''$ lie in $A\backslash B$. Further, equip $X$ with a well order $\leq$ as follows.  First fix a well order $\leq'_r$ on $X'_r:=\{~x\in X'~|~ \deg(x) =r ~\}$ and a well order $\leq''_r$ on $X''_r:=\{~x\in X''~|~ \deg(x) =r ~\}$ for each integer $r\geq 1$; then for $x_1,x_2\in X$,
	\begin{eqnarray*}
		\label{definition-deglex}
		x_1 \leq x_2\,\, \Longleftrightarrow \,\, \left\{
		\begin{array}{llll}
			x_1\in X' \text{ and } x_2\in X'', \quad \text{or} \\
			x_1, x_2\in X' \text{ and }	 \left\{
			\begin{array}{llll}
				\deg(x_1)<\deg(x_2), \quad \text{or} && \\
				\deg(x_1) =\deg(x_2)\, \text{ and } x_1\leq'_{\deg(x_1)} x_2,
			\end{array}\right. \text{or} \\
			x_1, x_2\in X'' \text{ and }	 \left\{
			\begin{array}{llll}
				\deg(x_1)<\deg(x_2), \quad \text{or} && \\
				\deg(x_1) =\deg(x_2)\, \text{ and } x_1\leq''_{\deg(x_1)} x_2,
			\end{array}\right.
		\end{array}\right.
	\end{eqnarray*}
Since $\Delta$ is a pseudo-comultiplication on $A$ such that $\Delta(B) \subseteq B\otimes B$,  one may readily lift $\Delta$ to a triangular comultiplication $\tilde{\Delta}: k\langle X\rangle \to k\langle X\rangle \otimes k\langle X\rangle$ such that 
	\[
	\tilde{\Delta}(k\langle X'\rangle) \subseteq k\langle X'\rangle \otimes k\langle X'\rangle \quad \text{ and }\quad \tilde{\Delta}(I) \subseteq k\langle X\rangle \otimes I + I \otimes k\langle X\rangle.
	\]
	
	Clearly, the canonical projection from $k\langle X\rangle$ to $A$ induces a graded algebra isomorphism $$k\langle X'\rangle/J\cong B, \quad \text{ where }J=I\cap k\langle X'\rangle.$$
	Let $\Gamma:= \mathcal{N}_I\backslash \mathcal{N}_J$ and define $\leq$ to be the restriction of $\leq_{\lex}$ to $\Gamma$. 
	For every $\gamma\in \Gamma$, let
	\[
	z_\gamma:= [\gamma]+I.
	\]  
	Since $\mathcal{N}_J = \mathcal{N}_I\cap \langle X'\rangle$ by Lemma \ref{Lyndon-compare},
	the left-most letter of every Lyndon word in $\Gamma$ must belong to $X''$. Therefore,
	every Lyndon word in $\mathcal{N}_J$ is lexicographically smaller than that in $\Gamma$. Note that $B$ is generated by $\{~ [v]+I~|~v\in \mathcal{N}_J~\}$, hence Parts (1) and (2) follow  by Proposition \ref{quasi-Lie-imply-LK} (3, 4). Compare the basis of $A$ and $B$ that constructed from $\mathcal{N}_I$ and $\mathcal{N}_J$ respectively as in Proposition \ref{quasi-Lie-imply-LK} (5), it is easy to see that $A$ is a free left $B$-module with a basis as the given set in Part (3). Since $\gkdim A=\# (\mathcal{N}_I)$ and $\gkdim B = \#(\mathcal{N}_J)$ by Proposition \ref{quasi-Lie-imply-LK} (6), Part (4) follows directly.
	
	Next we  show that  $A$ is generated  as a right $B$-module by the  set given in Part (3). For $m\geq 1$, let $\Gamma_{m}$ be the set of all nondecreasing finite sequences $(\gamma_1,\ldots, \gamma_q)$ in $\Gamma$  with $\sum_{i=1}^q \deg(z_{\gamma_i}) \leq m$. Clearly,
	\[
	A_m= \bigoplus_{(\gamma_1,\ldots, \gamma_q) \in \Gamma_m} A_m^{(\gamma_1,\ldots, \gamma_q)}, \text{ where } A_m^{(\gamma_1,\ldots, \gamma_q)}:= B_{m-\sum_{i=1}^q \deg(z_{\gamma_i})} \cdot( z_{\gamma_1}\cdots z_{\gamma_q}).
	\]
	Let  $\leq_m$ be the total order on $\Gamma_{m}$ defined as follows.  For $\overline{\gamma}= (\gamma_1,\ldots, \gamma_q)$ and $\overline{\delta}=(\delta_1,\ldots, \delta_p)$ in $\Gamma_{m}$,  $\overline{\gamma}<_m \overline{\delta}$ if and only if $\overline{\gamma}$ is a proper right part of $\overline{\delta}$ or  there  exists an integer $l\leq \min\{p,q\}$ such that $\gamma_{q-i}=\delta_{p-i}$ for $i<l$ and $\gamma_{q-l}< \delta_{p-l}$. By a similar argument as that in Lemma \ref{lex-stablize}, it is not hard to see $\leq_m$ is a well-order on $\Gamma_{m}$ with the empty sequence $\emptyset$ as the smallest element. Note that 
	\[
	A_m^\emptyset = B_m.
	\]
	For any sequence $\overline{\gamma}=(\gamma_1,\ldots, \gamma_s, {\small\overbrace{\gamma,\ldots, \gamma}^{p})} \in \Gamma_m $, where $\gamma_s<\gamma$ and $p\geq 1$, and for any $\lambda\in B_{m-\deg(z_{\gamma_1}\cdots z_{\gamma_s}z_\gamma^p)}$, one has by Part (2) and Lemma  \ref{irreducible-Lyndon-subalgebra} that 
	\begin{eqnarray*}
	[\lambda, z_{\gamma_1}\cdots z_{\gamma_s}z_\gamma^p] &=&   [\lambda, z_{\gamma_1}\cdots z_{\gamma_s}]z_\gamma^p + z_{\gamma_1}\cdots z_{\gamma_s}\cdot [\lambda,z_\gamma^p] \\ &\in&  [\lambda, z_{\gamma_1}\cdots z_{\gamma_s}]z_\gamma^p +  \sum_{i=0}^{p-1} A^{<\gamma}\cdot z_\gamma^i \\ &\subseteq &\sum\nolimits_{(\eta_{1},\ldots, \eta_l, {\tiny\overbrace{\gamma,\ldots, \gamma}^{p})})}A_m^{(\eta_1,\ldots, \eta_l, {\tiny\overbrace{\gamma,\ldots, \gamma}^{p})}}  + \sum\nolimits_{(\delta_{1},\ldots, \delta_t, {\tiny\overbrace{\gamma,\ldots, \gamma}^{i})})}A_m^{(\delta_1,\ldots, \delta_t, {\tiny\overbrace{\gamma,\ldots, \gamma}^{i})}},
	\end{eqnarray*}
	where $(\eta_1,\ldots, \eta_l)$ runs over elements in $\Gamma_{m-\deg(z_\gamma^p)}$ such that  $(\eta_1,\ldots, \eta_l)<_m $ $(\gamma_1,\ldots, \gamma_s)$, and where   $(\delta_1,\ldots, \delta_t, {\small\overbrace{\gamma,\ldots, \gamma}^{i})} $ runs over elements in $\Gamma_m$ with $i<p$ and $\delta_t<\gamma$. It follows that 
	\begin{eqnarray}\label{key formula}
		[\lambda, z_{\gamma_1}\cdots z_{\gamma_s}z_\gamma^p] \in \sum_{\overline{\delta}\in \Gamma_m,\, \overline{\delta}  <_m \overline{\gamma}} A_m^{\overline{\delta}},
	\end{eqnarray}
	because of
	\[
(\eta_1,\ldots, \eta_l, {\small\overbrace{\gamma,\ldots, \gamma}^{p})}, \quad 	(\delta_1,\ldots, \delta_t, {\small\overbrace{\gamma,\ldots, \gamma}^{i})} \, <_m \, \overline{\gamma}= (\gamma_1,\ldots, \gamma_s, {\small\overbrace{\gamma,\ldots, \gamma}^{p})}.
	\] 
By induction on $\Gamma_m$ with respect to $\leq_m$, the space $A_m^{(\gamma_1,\ldots, \gamma_q)}$ is contained in the span of the given set in Part (3) as a right $B$-module for all $(\gamma_1,\ldots, \gamma_q)\in \Gamma_m$, and therefore $A= \sum_{m\geq 0} A_m$ is too.
	
Finally, we show the given set in Part (3) is right linearly independent over $B$. Otherwise, there exist positive integers $m,\, n\geq 1$, indexes $\overline{\gamma(1)} <_m\cdots <_m \overline{\gamma(n)} \in \Gamma_m$ with $\overline{\gamma(i)}=(\gamma(i)_1,\ldots, \gamma(i)_{s(i)})$ and nonzero homogeneous elements $\lambda(1),\ldots, \lambda(n) \in B$ with $\deg(\lambda(i)) = m- \sum_{j=1}^{s(i)}\deg(z_{\gamma(i)_j})$ such that  
		\begin{equation*}\label{linear-relation}
		\sum _{i=1}^n( z_{\gamma(i)_1}\cdots z_{\gamma(i)_{s(i)}}) \cdot \lambda(i) =0.
	\end{equation*}
By the formula (\ref{key formula}) established in the previous paragraph, one has
\[
(z_{\gamma(i)_1}\cdots z_{\gamma(i)_{s(i)}}) \cdot \lambda(i) \in  \lambda(i) \cdot (z_{\gamma(i)_1}\cdots z_{\gamma(i)_{s(i)}}) 
+\sum_{\overline{\delta}\in \Gamma_m,\, \overline{\delta}<_m \overline{\gamma(i)}} A_m^{\overline{\delta}}, \quad \quad i=1,\cdots, n.
\]
It follows that 
\begin{eqnarray*}
0 &\in& 	\sum _{i=1}^n\lambda(i) \cdot ( z_{\gamma(i)_1}\cdots z_{\gamma(i)_{s(i)}}) +\sum_{\overline{\delta}\in \Gamma_m,\, \overline{\delta}<_m \overline{\gamma(n)}} A_m^{\overline{\delta}}  \\ &=& \lambda(n) \cdot ( z_{\gamma(n)_1}\cdots z_{\gamma(n)_{s(n)}}) +\sum_{\overline{\delta}\in \Gamma_m,\, \overline{\delta}<_m \overline{\gamma(n)}} A_m^{\overline{\delta}}.
\end{eqnarray*}
Since  the given set in Part (3) is left $B$-linearly independent, one may readily conclude that $\lambda(n)=0$, which is a contradiction. This completes the proof of Part (3).
\end{proof}

\section{Applications to connected (graded) Hopf algebras}

\label{section-Hopf-algebra}

In this section we study the structure of connected (graded) Hopf algebras.

Let us begin by recalling some notations and definitions on Hopf algebras.  For a general Hopf algebra $H$, the usual notations $\Delta_H$, $\varepsilon_H$ and $S_H$ are employed to denote  the comultiplication,  counit  and  antipode of $H$ respectively. The \emph{coradical} of a Hopf algebra $H$ is defined to be the sum of all simple subcoalgebras of $H$. It is denoted by $H_{(0)}$. Also,  the \emph{coradical filtration} of $H$ (\cite[Section 5.2]{Mont}) is denoted by $\{H_{(n)}\}_{n\geq0}$.  Note that the notations for coradical and coradical filtration we used  differ from that of \cite{Mont}. A Hopf algebra is called \emph{connected} if  its coradical  is  one-dimensional.

By a \emph{graded Hopf algebra} we mean a Hopf algebra $H$ equipped with a grading $H=\bigoplus_{n\geq 0}H_n$  such that $H$ is both a graded algebra and a graded coalgebra, and the antipode  preserves the given grading. 
Such a graded Hopf algebra is called \emph{connected} if $H_0=k$.  It is easy to show that, when $H$ is a connected graded Hopf algebra, one has $\ker \varepsilon_H= H_+:=\bigoplus_{n\geq 1}H_n$ and $$\Delta_H(x)\in x\otimes 1 + 1\otimes x + (H_+\otimes H_+)_n, \quad x\in H_n.$$
So the comultiplication of a connected graded Hopf algebra is a pseudo-comultiplication. Clearly connected graded Hopf algebras are connected Hopf algebras. For a connected Hopf algebra $H$, the coradical filtration $\{H_{(n)}\}_{n\geq 0}$ is a Hopf algebra filtration \cite[p. 62]{Mont}, whence
the associated graded space $\gr_c(H):= \bigoplus_{n\geq 0} H_{(n)}/H_{(n-1)}$ is a connected graded Hopf algebra in the natural way.

\begin{theorem}\label{generator-connected-graded-Hopf-algebra}
Let $H$ be a connected graded Hopf algebra and $K$ a graded Hopf subalgebra. Then there is an indexed family $(z_\gamma)_{\gamma\in \Gamma}$ of homogeneous elements of $H$ and a total order $\leq$ on $\Gamma$ such that
\begin{enumerate}
	\item $\Delta(z_\gamma)  \in 1\otimes z_\gamma +z_\gamma\otimes 1 + (H_+^{<\gamma} \otimes H_+^{<\gamma})_{\deg(z_\gamma)}$ for every index  $\gamma \in \Gamma$, where $H^{<\gamma}$ is the subalgebra of $H$ generated by $K$ and $\{~ z_\delta~|~ \delta \in \Gamma,~ \delta<\gamma ~\}$;
	\item $[z_\gamma , H^{<\gamma}] \subseteq H^{<\gamma}$ for every index  $\gamma \in \Gamma$;
	\item $\{~ z_{\gamma_1} \cdots z_{\gamma_n} ~|~ n\geq0,~ \gamma_1,\ldots, \gamma_n \in \Gamma, ~ \gamma_1\leq \cdots \leq \gamma_n ~\}$ is a homogeneous basis of $H$  as a left graded $K$-module as well as a right graded $K$-module;
	\item $\gkdim H = \gkdim K + \#(\Gamma)$, where $\#(\Gamma)$ is the number of elements of $\Gamma$.
\end{enumerate}
\end{theorem}

\begin{proof}
It is a direct consequence of Theorem \ref{generator-connected-graded-algebra}.
\end{proof}


\begin{corollary}\label{polynomial-generator}
Let $H$ be a commutative connected graded Hopf algebra and $K$ a graded Hopf subalgebra. Then $H$ is isomorphic as a graded algebra to the graded polynomial algebra over $K$ in some family of graded variables.
\end{corollary}

\begin{proof}
It is a direct consequence of Theorem \ref{generator-connected-graded-Hopf-algebra} (3).
\end{proof}

\begin{corollary}\label{graded-commutative-Hopf-algebra}
Let $H$ be a commutative Hopf algebra which is connected graded as an algebra. Let $K$ be a Hopf subalgebra as well as a graded subalgebra. If $(H_+)^2\cap K= (K_+)^2$ then $H$ is isomorphic as a graded algebra to the graded polynomial algebra over $K$ in some family of graded variables.
\end{corollary}

\begin{proof}
Let $\mathfrak{m}:=H_+= \sum_{n\geq 1} H_n$ and $\mathfrak{n} :=K_+= \sum_{n\geq 1} K_n$.  By the argument of \cite[Lemma 2.1 (2)]{BGZ}, to see the result we may assume $\mathfrak{m} =\ker (\varepsilon_H)$, in which case $\mathfrak{n}=\ker(\varepsilon_K)$. By \cite[Lemma 3.3]{GZ}, $A= \bigoplus_{i\geq 0} \mathfrak{m}^i/\mathfrak{m}^{i+1} $ and  $B= \bigoplus_{i\geq 0} \mathfrak{n}^i/\mathfrak{n}^{i+1} $ are connected graded Hopf algebras with comultiplication and counit  defined in \cite[Lemma 3.2]{GZ}. Since $A$ (resp. $B$) is generated by $A_1=\mathfrak{m}/\mathfrak{m}^2$ (resp. $B_1=\mathfrak{n}/\mathfrak{n}^2$) and commutative, Theorem \ref{generator-connected-graded-Hopf-algebra} (3) tells us that $A$ (resp. $B$) is freely generated as a commutative algebra on any choice of bases for $A_1=\mathfrak{m}/\mathfrak{m}^2$ (resp. $B_1=\mathfrak{n}/\mathfrak{n}^2$). Assume $\mathfrak{m}^2\cap K = \mathfrak{n}^2$. Then the canonical map $B_1\to A_1$ is injective. Fix a basis $X$ for $B_1$ and extend it to a basis $Y$ for $A_1$ under the canonical map $B_1\to A_1$. Then the canonical map $B\to A$ sends different commutative monomials in $X$ to different commutative monomials in $Y$. Since the commutative monomials in  $X$ (resp. $Y$) form a basis for $B$ (resp. $A$), the canonical map $B\to A$ is an injective homomorphism of graded Hopf algebras.  
Therefore, we may consider $B$ as a graded Hopf subalgebra of $A$. Note that $A$ is actually $\mathbb{Z}^2$-graded as an algebra with $A_{(i,j)} = (\mathfrak{m}^i/\mathfrak{m}^{i+1})_j$, and $B$ a $\mathbb{Z}^2$-graded subalgebra. Choose an indexed family of homogeneous elements $\{a_\gamma\}_{\gamma\in \Gamma}$ of $\mathfrak{m}$ such that the family of cosets  $\{~ \overline{a_\gamma} ~ \}_{\gamma\in \Gamma}$, where $\overline{a_\gamma}:=a_\gamma+\mathfrak{m}^2$, forms a  basis of a complement of $B_1$ in $A_1$. Fix a total order $\leq$ on $\Gamma$. Then the set
$$\{~ \overline{a_{\gamma_1}} \cdots \overline{a_{\gamma_n}} ~|~ n\geq0,~ \gamma_1,\ldots, \gamma_n \in \Gamma, ~ \gamma_1\leq \cdots \leq \gamma_n ~\}$$
is a homogeneous basis of  $A$ as a $\mathbb{Z}^2$-graded $B$-module. By a standard application of filtered-graded method and the observation that $(\mathfrak{m}^i)_r = 0$ for integers $0\leq r<i$, it is not hard to conclude that
the set $$\{~ a_{\gamma_1} \cdots a_{\gamma_n} ~|~ n\geq0,~ \gamma_1,\ldots, \gamma_n \in \Gamma, ~ \gamma_1\leq \cdots \leq \gamma_n ~\}$$ forms a homogeneous basis of $H$ as a graded $K$-module. Since $H$ is commutative,  $H$ is a graded  polynomial algebra over $K$ in some family of graded variables.
\end{proof}

Let $H$ be a Hopf algebra and $K$ a Hopf subalgebra. For an element $a\in H$, let $\rho_H(a)$ be the smallest number  $n$ such that $a\in H_{(n)}$. By a \emph{filter basis} of $H$ as a left (resp. right) $K$-module we mean a basis  $(a_i)_{i\in I}$ of $H$ as a left (resp. right) $K$-module such that 
\[
H_{(n)} = \sum_{i\in I} K_{(n-\rho_H(a_i))} \cdot a_i \quad  (\text{resp. } H_{(n)} = \sum_{i\in I}  a_i\cdot K_{(n-\rho_H(a_i))} ), \quad n\geq 0.
\]

\begin{corollary}\label{generator-connected-Hopf-algebra}
Let $H$ be a connected Hopf algebra and $K$ a Hopf subalgebra. Then there is an indexed family $(x_\gamma)_{\gamma\in \Gamma}$ of elements of $\ker(\varepsilon_H)$ and a total order $\leq$ on $\Gamma$ such that $\gkdim H = \gkdim K + \#(\Gamma)$, where $\#(\Gamma)$ is the number of elements of $\Gamma$, and $\{~ x_{\gamma_1} \cdots x_{\gamma_n} ~|~ n\geq0,~ \gamma_1,\ldots, \gamma_n \in \Gamma, ~ \gamma_1\leq \cdots \leq \gamma_n ~\}$ is a filter basis of $H$ as a left $K$-module as well as a right $K$-module.
\end{corollary}

\begin{proof}
By \cite[Lemma 5.2.12]{Mont}, we may consider $\gr_c (K)$ as a Hopf subalgebra of $\gr_c (H)$ in the natural way. Choose an indexed family $\{z_\gamma\}_{\gamma\in \Gamma}$ of homogeneous elements of $\gr_c (H)$ of positive degrees and   a total order $\leq$ on $\Gamma$ as in Theorem \ref{generator-connected-graded-Hopf-algebra} for the pair $(\gr_c(H), \gr_c(K))$. Then, for each index $\gamma$ with $\deg(z_\gamma)=n_\gamma$, pick an element $x_\gamma\in H_{(n_\gamma)}$ such that $x_\gamma+H_{(n_\gamma-1)} =z_\gamma$ in $\gr_c(H)$. Replacing $x_\gamma$ by $x_\gamma- \varepsilon_H(x_\gamma)$, one may assume $x_\gamma\in \ker(\varepsilon_H)$. Note that $\rho_H(x_\gamma)= n_\gamma$ and $\rho_H(x_{\gamma_1}\cdots x_{\gamma_n}) = \rho_H(x_{\gamma_1}) +\cdots +\rho_H(x_{\gamma_n})$.  By  \cite[Lemma I.5.1 (3) ]{NaVa}, the last part follows directly. In addition, by \cite[Proposition 3.4]{ZSL3} and  Theorem \ref{generator-connected-graded-Hopf-algebra} (4), one has $\gkdim H= \gkdim \gr_c(H) = \gkdim \gr_c(K) + \#(\Gamma) = \gkdim K+\#(\Gamma)$.
\end{proof}

\begin{corollary}\label{connected-commutative-Hopf-algebra}
Let $H$ be a commutative connected Hopf algebra and $K$ a Hopf subalgebra. Then  $H$ is isomorphic as an algebra to the polynomial algebra over $K$ in some family of variables.
\end{corollary}

\begin{proof}
It is a direct consequence of Corollary \ref{generator-connected-Hopf-algebra}.
\end{proof}

Now we turn to consider the structure of connected (graded) Hopf algebras of finite GK dimension. First we recall the definition of Hopf Ore extensions and iterated Hopf Ore extensions.
\begin{definition}\label{definition-Hopf-Ore}
Let $R$ be a Hopf algebra and $n\geq 2$. 
\begin{enumerate}
	\item A Hopf algebra $H$ is called a \emph{Hopf Ore extension} (\emph{HOE}, for short) of $R$ if it contains $R$ as a Hopf subalgebra and  is an Ore extension of $R$ as algebras, that is $H=R[X;\sigma,\delta]$ for some algebra automorphism $\sigma$ of $R$ and $\sigma$-derivation $\delta$ of $R$. We say that the HOE $H$ is \emph{of derivation type} if $\sigma$ can be chosen to be the identity map of $R$. 
    \item A Hopf algebra $H$ is called an \emph{($n$-step) iterated Hopf Ore extension} (\emph{($n$-step) IHOE}, for short) of $R$ if it contains $R$ as a Hopf subalgebra and there is a chain of Hopf subalgebras \[R=H^0\subset H^1 \subset \cdots \subset H^n= H\]such that $H^i$ is an Ore extension of $H^{i-1}$ as algebras for $i=1,\ldots, n$. We say that the IHOE $H$ is \emph{of derivation type} if all successive HOEs'  can be chosen to be  of derivation type .
\end{enumerate}
By convention, a $0$-step IHOE of $R$ is just $R$ itself; a $1$-step IHOE of $R$ is just a HOE of $R$. One similarly define the concept of  \emph{graded HOE} and \emph{graded $n$-step IHOE} for graded Hopf algebras.
\end{definition}

\begin{theorem}\label{IHOE-graded}
Let $H$ be a connected graded Hopf algebra of GK dimension $d$ and $K$ a graded Hopf subalgebra of GK dimension $e$. Then $H$ is a graded $(d-e)$-step  IHOE of $K$ of derivation type.
\end{theorem}

\begin{proof}
Let $l=d-e$. Fix an  indexed family $\{z_\gamma\}_{\gamma\in \Gamma}$ of homogeneous elements of $H$ of positive degrees and a total order $\leq$ on $\Gamma$ as in Theorem \ref{generator-connected-graded-Hopf-algebra}. Since $\#(\Gamma)=l$, we may assume $\Gamma= \{~1,\ldots, l ~\}$ equipped with the natural order. Let $H^0= K$; let $H^i$ be the subalgebra of $H$ generated by $K$ and $z_1,\ldots, z_i$ for $i=1,\ldots, l$. Clearly,  $H^i$ are all  Hopf subalgebras of $H$ by Theorem \ref{generator-connected-graded-Hopf-algebra} (1) and induction on $i$. Also, one may readily conclude by Theorem \ref{generator-connected-graded-Hopf-algebra} (2, 3) and induction on $i$ that the set \[\{~ 1, ~ z_i, ~ z_i^2, \ldots ~\}\] is a homogeneous basis of $H^i$ as a graded left $H^{i-1}$-module for $i=1,\ldots, d$. In addition, according to Theorem \ref{generator-connected-graded-Hopf-algebra} (2), one may define a graded map $\delta_i:H^{i-1} \to H^{i-1}$ of degree $\deg(z_i)$, $i=1,\ldots, d$, by $$f\mapsto z_i \cdot f -f\cdot z_i, \quad f\in H^{i-1}.$$ It is easy to check that $\delta_i$ is a derivation of $H^{i-1}$  and $H^{i} = H^{i-1} [z_i;\delta_i].$
\end{proof}

\begin{corollary}
Let $H$ be a connected graded Hopf algebra of GK dimension $d$.  Then $H$ is a $(d-e)$-step graded IHOE of the universal enveloping algebra of $P(H)$ of derivation type, where $P(H)$ is the primitive space of $H$ and $e$ is the dimension of $P(H)$.
\end{corollary}

\begin{proof}
Note that the universal  enveloping algebra of $P(H)$  is canonically  a graded Hopf subalgebra of $H$ and of GK dimension $e$. The result then follows directly by Theorem \ref{IHOE-graded}.
\end{proof}

\begin{proposition}
Let $H$ be a connected Hopf algebra of GK dimension $d$ and $K$ a Hopf subalgebra of GK dimension $d-1$. Let $r$ be the smallest number such that $H_{(r)}\neq K_{(r)}$. Assume that $K$ is generated by $K_{(r-1)}$ as an algebra. Then $H$ is a Hopf Ore extension of $K$.
\end{proposition}

\begin{proof}
By Corollary \ref{generator-connected-Hopf-algebra}, there's an element $x\in H$ satisfying the following conditions:
\begin{enumerate}
	\item  $\{1,x,x^2,\ldots\}$ forms a basis of $H$ as a left  $K$-module; 
	\item  $\{1,x,x^2,\ldots\}$ forms a basis of $H$ as a  right $K$-module; 
	\item  $H_{(n)} = \sum_{i\geq 0}K_{(n-\rho_H(x)\cdot i)}\cdot x^i = \sum_{i\geq 0}x^i\cdot  K_{(n-\rho_H(x)\cdot i)}$ for $n\geq 0$.
\end{enumerate}
Then clearly, $\rho_H(x) =r$, 
\begin{eqnarray*}
x\cdot K_{(r-1)}  \subseteq H_{(2r-1)} =  K_{(r-1)}\cdot x +K_{(2r-1)} \quad {and}\quad K_{(r-1)} \cdot x  \subseteq H_{(2r-1)} =  x\cdot K_{(r-1)}+K_{(2r-1)}.
\end{eqnarray*}
Since $K$ is generated by $K_{(r-1)}$, it follows that 
\begin{eqnarray*}
	x\cdot K  \subseteq   K\cdot x +K \quad {and} \quad K \cdot x  \subseteq  x\cdot K+K.
\end{eqnarray*}
By (1) and  $x\cdot K  \subseteq   K\cdot x +K$, one obtains that $H=K[x;\sigma, \delta]$ for some endomorphism $\sigma$ of $K$ and $\sigma$-derivation $\delta$ of $K$. It remains to show $\sigma$ is an automorphism. If $\sigma(a)=0$ then $x\cdot a = \delta(a)$, and hence $a=0$ by (2). So $\sigma$ is injective. Since $K \cdot x  \subseteq  x\cdot K+K$,  for any $a\in K$ one has \[a\cdot x = x \cdot a_1 + a_2 = \sigma(a_1) \cdot x +\delta(a_1)+a_2\] for some $a_1,a_2\in K$. Then $a=\sigma(a_1)$ by (1). So $\sigma$ is surjective.
\end{proof}

\vskip7mm

\noindent{\it Acknowledgments.}  
The second-named author would like to thank Houyi Yu for helpful discussions, particularly for calling his attention to \cite{MaRe}. G.-S. Zhou is supported by the NSFC (Grant Nos. 11971141 \& 11871186), the Fundamental Research Funds for the Provincial Universities of Zhejiang, and the K. C. Wong Magna Fund in Ningbo University. The authors thank the referee for his/her careful reading and valuable suggestions.

\end{document}